\documentclass[a4paper,11pt,twoside,reqno]{amsart}  
\DeclareSymbolFontAlphabet{\mathcal}{symbols}
\usepackage[colorlinks=red,linkcolor=blue]{hyperref}

\usepackage[all]{xy}

\usepackage{amsmath}
\usepackage[psamsfonts]{amssymb}
\usepackage[psamsfonts]{eucal}
\usepackage{amsthm}
\usepackage[scaled]{helvet}
\usepackage[a4paper]{geometry}

\usepackage{color}

\usepackage{enumerate}

\theoremstyle{plain}
\newtheorem{theorem}{Theorem}

\newtheorem{proposition}{Proposition}[section]

\newtheorem{lemma}[proposition]{Lemma}

\theoremstyle{definition}
\newtheorem{definition}[proposition]{Definition}

\theoremstyle{remark}
\newtheorem{remark}[proposition]{Remark}

\newcommand{\secref}[1]{Section~\ref{#1}}

\newcommand{\thmref}[1]{Theorem~\ref{#1}}
\newcommand{\propref}[1]{Proposition~\ref{#1}}

\newcommand{\defref}[1]{Definition~\ref{#1}}

\makeatletter
\renewcommand\l@subsection{\@tocline{2}{0pt}{2pc}{5pc}{}}
\renewcommand\l@subsubsection{\@tocline{3}{0pt}{4pc}{10pc}{}}
\makeatother

\def\R{{\mathbb R}}

\def\C{\mathbb{C}}

\def\HH{\mathbb{H}}

\def\R{\mathbb{R}}

\def\u{{\rm U}}

\def\id{{\rm id}}

\newcommand{\Sp}{\mathrm{Sp}}
\newcommand{\K}{\mathbb{K}}
\newcommand{\G}{\mathrm{G}}
\newcommand{\OO}{\mathrm{O}}
\newcommand{\U}{\mathrm{U}}
\newcommand{\XX}{\mathrm{X}}
\newcommand{\VV}{\mathrm{V}}
\newcommand{\WW}{\mathrm{W}}
\newcommand{\Hprod}[1]{\langle #1\rangle}
\newcommand{\cat}{\mathop{\mathrm{cat}}}		

\newcommand{\block}[2]{\addtolength{\arraycolsep}{-3pt}\begin{pmatrix}#1\cr#2\end{pmatrix}}
\makeatletter
\renewcommand*\env@matrix[1][*\c@MaxMatrixCols c]{%
  \hskip 0pt
  \let\@ifnextchar\new@ifnextchar
  \array{#1}}
\makeatother

\begin{document}


\title[Cayley Transform on  Stiefel manifolds]{Cayley Transform on  Stiefel manifolds}

\author[E. Mac\'{\i}as-Virg\'os]{Enrique Mac\'{\i}as-Virg\'os}
\address{Departamento de Matem\'aticas\\
Universidade de Santiago de Compostela\\
15782 Spain}
\email{quique.macias@usc.es}

\author[M.-J. Pereira-S\'aez]{Mar\'{\i}a~Jos\'e~Pereira-S\'aez}
\address{Facultade de Econom\'{\i}a e Empresa, 
		Universidade da Coru\~na\\
		15071
		Spain.}
\email{maria.jose.pereira@udc.es}

\author[D. Tanr\'e]{Daniel Tanr\'e}
\address{D\'epartement de Math\'ematiques \\
Universit\'e de Lille 1 \\
59655 Villeneuve d'Ascq Cedex, France}
\email{Daniel.Tanre@univ-lille1.fr}

\date{\today}

\begin{abstract} 
We define a Cayley transform on  Stiefel manifolds. 
Applications to the  Lusternik-Schnirelmann category and optimisation problems are presented.
\end{abstract}




\maketitle



\section*{Introduction}
Denote by $\K$  the algebra of either  the real numbers $\R$, the complex numbers $\C$  or  the quaternions $\HH$.
Let  
$\G(n)=\OO(n,\K)$ be the  Lie group of matrices $A\in \K^{n\times n}$ such that
$AA^*=I_n$, where $A^*=\bar{A^t}$ is the conjugate transpose. 
Depending on $\K$ this group corresponds to the orthogonal group $\OO(n)$, 
the unitary group $\U(n)$ or the symplectic group $\Sp(n)$.

Let $I\in \K^{n\times n}$ be the identity matrix.
The classical Cayley transform, $c_{I}\colon \Omega({I})\to \Omega({I})$  is defined by
$c_{I}(X) =(I-X)(I+X)^{-1}$
with $\Omega({I})=
\left\{ X \in \K^{n\times n}\mid (I+X)^{-1} \text{ exists}\right\}$.
This map satisfies the equality $c_{I}^2=\id$.  Moreover $c_I$  induces a diffeomorphism between the tangent space 
$T_{I}\G(n)=\{X\in \K^{n\times n}\mid X+X^*=0\}$  and 
$\Omega(I)\cap \G(n)$.
This construction was generalized by 
A.~G\'omez-Tato and the  first two authors \cite{GomMacPer2011} to any $A\in\G(n)$ as a map 
$c_{A}\colon \Omega(A)\to \Omega(A^*)$ defined by
\begin{equation}\label{eq:tato}
c_A(X)=(I-A^*X)(A+X)^{-1} =c_I(A^*X)A^*,
\end{equation}
with 
$$\Omega(A)=\left\{X\in\K^{n\times n}\mid (A+X)^{-1} \text{ exists}\right\}.$$
In this case we have $c_A^{-1}=c_{A^*}$ and there is a diffeomorphism between the tangent space
$T_A\G(n)=\left\{
X\in\K^{n\times n} \mid A^*X+X^*A=0\right\}$
and 
$\Omega(A^*)\cap \G(n)$.

In this work,   \emph{we construct Cayley transforms on Stiefel manifolds}. 
We first specify some conventions and notations in use in this paper and state our main results. 

Let $0\leq k\leq n$. 
The compact Stiefel manifold $ \OO_{n,k}$ of orthonormal $k$-frames in $\K^n$ is 
the set of matrices $x\in\K^{n\times k}$ such that $x^*x=I_k$. 
This manifold appears also  as the basis of the principal fibration
$$
\G(n-k)\stackrel{\iota}{\longrightarrow}
\G(n)\stackrel{\rho}{\longrightarrow}
\OO_{n,k},
$$
where $\iota(B)=\begin{pmatrix}
B&0\cr
0&I_{k}
\end{pmatrix}
$ and $\rho$ is the projection onto the last $k$ columns.
If $A\in\G(n)$ and $x=\rho(A)$, 
we denote by $\rho_{*A}\colon T_A\G(n)\to T_{x}\OO_{n,k}$ the map induced  
between the tangent spaces. 

\smallskip

The next statement contains the existence and the main properties of a Cayley transform in Stiefel manifolds.

\begin{theorem}\label{thm:propcayley}
Let $0\leq k\leq n$ and $x=\block{T}{P}\in \OO_{n,k}$ with $P\in\K^{k\times k}$. We choose $A=
\begin{pmatrix}
\alpha&T\cr
\beta&P
\end{pmatrix}\in\G(n)$.
Then there exists a map
$$\gamma^A\colon T_x\OO_{n,k}\to \OO_{n,k}$$
such that $\gamma^A\circ \rho_{*A}=\rho\circ c_A$. Moreover we have
 the following properties.
\begin{enumerate}[1)]
\item The map $\gamma^A$ is injective on the open subset
$$\Gamma^x=\left\{
v=A\block{X}{Y}\in T_x \OO_{n,k} \mid
(\beta X+P)^{-1} \text{ exists}\right\}. 
$$
This subset $\Gamma^x$ does not depend on the choice of $A$ such that $\rho(A)=x$. Furthermore, 
 if $\gamma^A$ is injective on an open subset
$U\subset T_x\OO_{n,k}$
then we have $U\subset \Gamma^x$.
\item The map $\gamma^A$ induces a diffeomorphism between $\Gamma^x\subset T_x \OO_{n,k}$ and the open subset
$$\Omega^x  
=\left\{
\block{\tau}{\pi}\in \OO_{n,k} \mid (\pi+P^*)^{-1} \text{ exists}
\right\}.$$
\end{enumerate}
\end{theorem}

An explicit formula for $\gamma^A$ is given in \defref{def:cayley}. Also, the expression of the inverse map $({\gamma^A}_{\mid\Gamma_x})^{-1}$ appears in Equations \eqref{equa:leX} and \eqref{equa:ygrecq}. 

\smallskip

As we said before for the group $\G(n)$, the Cayley transform 
$c_A\colon T_A\G(n)\to \Omega(A^*)\cap \G(n)$
is a diffeomorphism. Therefore the Cayley open subset $\Omega(A^*)\cap \G(n)$ is  contractible.
This property cannot be extended as it stands in the case of a Stiefel manifold.
However the image of the injectivity domain of a Cayley transform in $\OO_{n,k}$ is contractible in 
$\OO_{n,k}$.

\begin{theorem}\label{thm:cayleycontraction}
For every $x\in \OO_{n,k}$ the open subset
$\Omega^x  $
is contractible in $\OO_{n,k}$.
\end{theorem}

This property is a consequence of the existence of a local section 
(see \propref{prop:localsection})
$s^A\colon \Omega^x\to \G(n)$
of the projection $\rho\colon \G(n)\to \OO_{n,k}$
and the contractibility of the Cayley open subsets $\Omega(A^*)\cap \G(n)$.

\smallskip

The contents of the paper are as follows. \secref{sec:CayleyStiefel} contains the construction of the Cayley transform
$\gamma^A\colon T_x\OO_{n,k}\to \OO_{n,k}$.
The study of the injectivity of its derivative is done in \secref{sec:derivative} and the proofs of 
Theorems \ref{thm:propcayley} and \ref{thm:cayleycontraction}
occupy \secref{subsec:propcayley}.
Finally, \secref{sec:appli}
is devoted to applications of this construction to  Lusternik-Schnirelmann category
of the quaternionic Stiefel manifolds and to optimisation problems on real Stiefel manifolds. 


\section{Construction}\label{sec:CayleyStiefel}
Let $\K^n$ be either the real vector space $\R^n$, the complex vector space $\C^n$ or the quaternionic vector space $\HH^n$   (with the structure of a right $\HH$-vector space) endowed with the inner product $\Hprod{u,v}=u^*v$. 
Let $0\leq k\leq n$. 
The compact Stiefel manifold $ \OO_{n,k}$ of orthonormal $k$-frames in $\K^n$ is 
the set of matrices $x\in\K^{n\times k}$ such that $x^*x=I_k$. 
It is standard to denote $\OO_{n,k}$ by  $\VV_{n,k}$ in the real case, 
 $\WW_{n,k}$ in the complex case and $\XX_{n,k}$ in the quaternionic case.

Usually we write $x=\block{T}{P}\in \OO_{n,k}$, with $T\in\K^{(n-k)\times k}$ and
$P\in \K^{k\times k}$. 
The linear left action of $\G(n)$ on $\OO_{n,k}$ is transitive and the isotropy group of 
$x_{0}=\block{0}{I_k}$
is isomorphic to $\G(n-k)$.
Therefore $\OO_{n,k}$ is diffeomorphic to $\G(n)/\G(n-k)$ and we have the principal fibration 
$\G(n-k)\xrightarrow[]{\iota} \G(n)\xrightarrow[]{\rho} \OO_{n,k}$.

Let  $x\in \OO_{n,k}$. We complete $x$ to a matrix $A\in\G(n)$  such that $\rho(A)=x$.
The tangent space $T_I\G(n)$ of the group $\G(n)$ at the identity is the set of 
skew-Hermitian (skew-symmetric in the real case) matrices and 
 the tangent space $T_A \G(n)$ at $A$ equals $A\cdot T_I\G(n)$.
On the other side, recall that
$$T_x\OO_{n,k}=\{v\in\K^{n\times k}\mid  v^*x+x^*v=0\}=A\cdot T_{x_0}\OO_{n,k},$$
where $x_0=\block{0}{I_k}=\rho(I_n)$.
So each tangent vector $v\in T_x\OO_{n,k}$ can be written as 
$v=A\block{X}{Y}$, with  $X\in\K^{(n-k)\times k}$,   
$Y\in \K^{k\times k}$ and $Y+Y^*=0$.
Grants to the principal fibration defining the Stiefel manifold,   the tangent space
$T_x \OO_{n,k}$ can be identified to the orthogonal   $\left(T_A\G(n-k)\right)^{\bot}$ of the image of the inclusion
of $\G(n-k)$ in $\G(n)$,
$$(T_A\G(n-k))^\perp
=
\left\{
A\begin{pmatrix}0&X\cr
-X^*&Y
\end{pmatrix}
\mid Y+Y^*=0\right\}
\cong 
T_x\OO_{n,k}.
$$
 With this identification, the tangent space 
 $T_x\OO_{n,k}=\left\{A\block{X}{Y}\mid Y+Y^*=0\right\}$
  is considered as a subspace of $T_A\G(n)$ and we may apply the Cayley map $c_A$ of $\G(n)$ on it. 
From Equation \eqref{eq:tato} we have
\begin{equation}\label{equa:cA}
c_A=R_{A^*}\circ c_I\circ L_{A^*},
\end{equation}
 where $L_{A^*}$ and $R_{A^*}$ denote as usual the  left and right multiplications in a Lie group.
 Thus, we have first to determine $c_I$ on the elements of 
 $\left(T_I\G(n-k)\right)^{\bot}$.

 \begin{lemma}\label{lemma:binverse} 
 Let $X\in \K^{(n-k)\times k}$ and $Y\in\K^{k\times k}$ such that $Y+Y^*=0$. Then the matrix $I_k+X^*X+Y$ is invertible.
 \end{lemma}
 
 \begin{proof}The skew-symmetric matrix $M=\begin{pmatrix}
 0&X\cr
 -X^*&Y
 \end{pmatrix}$ cannot have real eigenvalues, then $I_n+M=\begin{pmatrix}
I_{n-k}&X\cr
-X^*&I_k+Y
\end{pmatrix}
$ is invertible. In the following product,
\begin{equation}\label{eq:product}
\begin{pmatrix}
I_{n-k}&0\cr
X^*&I_k
\end{pmatrix}
\cdot
(I_n+M)=
\begin{pmatrix}
I_{n-k}&X\cr
0&X^*X+I_{k}+Y
\end{pmatrix},
\end{equation}
the two factors on the left-hand side
admit an inverse. So, the matrix on the right-hand side admits an inverse, and the results follows.
 \end{proof}
 We  denote
 \begin{equation}\label{eq:b}
 b=(I_k+X^*X+Y)^{-1}.
 \end{equation}

 \begin{proposition}\label{prop:cIonbot}
Let
 $M=\begin{pmatrix}
 0&X\cr
 -X^*&Y
 \end{pmatrix}
 \in T_I\G(n)$,
 then we have
\begin{equation}\label{equa:cIonbot}
c_I(M)
=
\begin{pmatrix}
I_{n-k}-2XbX^*&-2Xb\cr
2bX^*&-I_{k}+2b
\end{pmatrix}.
\end{equation}
 \end{proposition}
 
 \begin{proof}From Equation \eqref{eq:product} we have
 $$(I_n+M)^{-1}=\begin{pmatrix}
I_{n-k}&X\cr
0&b^{-1}
\end{pmatrix}^{-1} 
\begin{pmatrix}
I_{n-k}&0\cr
X^*&I_k
\end{pmatrix}=
\begin{pmatrix}
I_{n-k}&-Xb\cr
0&b
\end{pmatrix}
\begin{pmatrix}
I_{n-k}&0\cr
X^*&I_k\cr
\end{pmatrix}.$$
 By applying the definition of $c_I\colon T_I\G(n)\to \G(n)$, we get:
\begin{align*}
c_{I}
(M)&=
\left(I_{n}-
M
\right)
\left(I_{n}+
M
\right)^{-1}\\
&=
\begin{pmatrix}
I_{n-k}&-X\cr
X^*&I_{k}-Y
\end{pmatrix}
(I_n+M)^{-1} 
=
\begin{pmatrix}
I_{n-k}-2XbX^*&-2Xb\cr
2bX^*&-I_{k}+2b
\end{pmatrix}.\qedhere
\end{align*}
\end{proof}
 
 A computation from \eqref{equa:cA} and \eqref{equa:cIonbot} gives directly:

\begin{equation*}
c_A\left(A
M
\right)
=
\begin{pmatrix}
(I_{n-k}-2XbX^*)\alpha^*-2XbT^*
&
(I_{n-k}-2XbX^*)\beta^*-2XbP^*\cr
2bX^*\alpha^*+(-I_{k}+2b)T^*
&
2bX^*\beta^*+(-I_{k}+2b)P^*
\end{pmatrix}.
\end{equation*}

\medskip
The Cayley transform is now obtained by projecting this expression on $\OO_{n,k}$.

\begin{definition}\label{def:cayley}
Let $x=\block{T}{P}\in \OO_{n,k}$.
We choose
$A=
\begin{pmatrix}
\alpha&T\cr
\beta&P
\end{pmatrix}\in \G(n)$, 
and consider
$v=A\block{X}{Y}
\in T_x\OO_{n,k}$.
\emph{The Cayley transform on the Stiefel manifold,}
$\gamma^A\colon T_x\OO_{n,k}\to \OO_{n,k}$,
is defined by
\begin{align}
\gamma^A(v)=&
\begin{pmatrix}
(I_{n-k}-2XbX^*)\beta^*-2XbP^*\cr
2bX^*\beta^*+(-I_{k}+2b)P^*
\end{pmatrix}
=2\block{-Xb}{b}(\beta X+P)^*+\block{\beta^*}{\,-P^*},\label{equa:gammaA}
\end{align}
where $b$ is given in Equation \eqref{eq:b}.
\end{definition}

\begin{remark}\label{rem:choice}
The map $\gamma^A$ depends on the choice of $A$ such that $\rho(A)=x$. 
With the previous notation, the elements of $\G(n)$ that are sent on $x$ are the matrices
$A{\scriptstyle{\bullet}} E:=A \begin{pmatrix}
E&0\cr
0&I_k
\end{pmatrix}
=
\begin{pmatrix}
\alpha E&T\cr
\beta E& P
\end{pmatrix}$
with $E\in \G(n-k)$. 
We observe 
$v=
A\block{X}{Y}
=
(A{\scriptstyle{\bullet}} E )
\block{E^*X}{Y}$
and
$I_k+(E^*X)^*(E^*X)+Y=I_k+X^*X+Y$. 
Thus, in \eqref{equa:gammaA}, if we replace $X$ by $E^*X$,
$\beta$ by $\beta E$ and keep unchanged $b$ and $Y$, we get 
\begin{equation*}
\gamma^{A{\scriptstyle{\bullet}} E}(v)= 
\begin{pmatrix}E^*&0\cr
0&I_k\cr\end{pmatrix} \gamma^A(v).
\end{equation*}
\end{remark}

\medskip

We end this section by noticing that the behavior of $\gamma^A$ 
is different from that of   the Cayley transform
$c_A$  in $\G(n)$. For instance,  when $n-k\geq k$, if we choose $x=\block{T}{0}$ and 
$v= A
\block{0}{Y}$,  we have
$\gamma^A(v)=\block{\beta^*}{0}$, which does not depend on $Y$. Thus $\gamma^A$ is not injective on the tangent space $T_x \OO_{n,k}$.
We address the determination of a domain of injectivity for $\gamma^A$ in \secref{subsec:propcayley} but, before that,
we study the differential of $\gamma^A$. 

\section{Differential}\label{sec:derivative}
The results of this section are used in the study of the  domain of injectivity of the Cayley transform $\gamma^A$.

Let  $x=\block{T}{P}\in \OO_{n,k}$ and
$A=
\begin{pmatrix}
\alpha&T\cr
\beta&P
\end{pmatrix}\in\G(n)$,
as before.
Let $v_0=\block{X}{Y}\in T_{x_0}\OO_{n,k}$. The differential of $\gamma^A$ as a map $\gamma^A\colon T_x\OO_{n,k}\to \K^{n\times k}$
at the point $v=A
v_0\in T_x\OO_{n,k}$
is denoted
$$(\gamma^A)_{*v}\colon T_vT_x\OO_{n,k}\cong T_x\OO_{n,k}\to \K^{n\times k}.
$$
We compute
$(\gamma^A)_{*v}(w)$ for any 
$w=A
\block{M}{N}
\in T_x\OO_{n,k}$, that is, $N+N^*=0$.
Since, with the identification $T_x\OO_{n,k}\cong (T_A\G(n-k))^\perp$, we have
 \begin{equation}\label{eq:gamma}
 \gamma^A=\rho\circ R_{A^*}\circ c_I\circ L_{A^*},
 \end{equation} 
 the differential of $\gamma^A$ is determined by 
that of $c_I$.
Therefore, we first consider 
$$(c_I)_{*v_0} 
\block{M}{N} 
=
\frac{d}{dt}_{\left|{t=0}\right.}
c_I\left(\block{X}{Y} + t \block{\,M}{\,N} \right).
$$
Let
$$b_t^{-1}:=I_k+(X+tM)^*(X+tM)+Y+tN.$$
Its derivative  $\frac{d}{dt}_{\left|{t=0}\right.}b_t^{-1}$ is denoted $\xi$ and equals 
\begin{equation}\label{eq:xi}
\xi:=X^*M+M^*X+N.
\end{equation} 
Moreover $b_0^{-1}=b^{-1}$. From $b_t b_t^{-1}=I$, we deduce $b^\prime_0=-b\xi b$. Then, a direct computation from Equation \eqref{equa:cIonbot} gives
\begin{equation}\label{equa:ciderivative}
(c_I)_{*v_0} 
\block{M}{N}
=
\left(
\begin{array}{rr}
-2MbX^*+2Xb\xi bX^*-2XbM^*&\ \ 
-2Mb+2Xb\xi b\\
-2b\xi bX^*+2bM^*&
-2b\xi b
\end{array}\right).
\end{equation}

\begin{proposition}\label{prop:derivative}
With the previous notations, the differential of the Cayley transform of the Stiefel manifold $\OO_{n,k}$ is given by
\begin{equation}\label{equa:derivative}
(\gamma^A)_{*v}(w)=
\block
{(-2MbX^*+2Xb\xi bX^*-2XbM^*)\beta^*+
(-2Mb+2Xb\xi b)P^*}
{(-2b\xi bX^*+2bM^*)\beta^*-2b\xi bP^*}.
\end{equation}
\end{proposition}

\begin{proof}
The equality 
\eqref{eq:gamma} gives by the chain rule
$(\gamma^A)_{*v}(w)=\rho\left( (c_I)_{*v_0} (w_0)\cdot A^*\right)$, where 
$w_0=\block{M}{N}$ and 
$A^*=\begin{pmatrix}\alpha^*&\beta^*\cr
T^*&P^*\cr\end{pmatrix}$, because the projection $\rho$ and the translations  $R_{A^*}$ and $L_{A^*}$ are linear maps. Then
 formula \eqref{equa:ciderivative} gives the value (\ref{equa:derivative}).
 \end{proof}

\begin{proposition}The differential $(\gamma^A)_{*v}$ is injective if and only if the matrix $\beta X+P$ is invertible.
\end{proposition}

\begin{proof}According to \eqref{equa:derivative}, the kernel of $(\gamma^A)_{*v}$ is the space of solutions $\block{M}{N}$ of the system
$$\left\{
\begin{array}{lrcl}
\text{(i)}& 2Mb(\beta X+P)^*
&=&
2Xb\xi b(\beta X+P)^*-2XbM^*\beta^*,\\
\text{(ii)}&
\xi b(\beta X+P)^*
&=&
M^*\beta^*,
\end{array}
\right.
$$
where we have used that the matrix $b$ is invertible.
Then we get
$$
Mb(X^*\beta^*+P^*)=0,
$$
so the first system is equivalent to
$$\left\{
\begin{array}{lrcl}
\text{(iii)}& 
Mb(\beta X+P)^*
&=&
0,\\
\text{(iv)}&\xi b(\beta X+P)^*
&=&
(\beta M)^*.\\
\end{array}
\right.
$$
$\bullet$ If we suppose the matrix $\beta X+P$ invertible, 
then the equation (iii) gives $M=0$ and the equation (iv) gives $\xi=0$. 
Finally, from the definition of $\xi$ in Equation \eqref{eq:xi} we have $M=N=0$.

\medskip\noindent
$\bullet$
Conversely, we suppose  the kernel of $\beta X+P$  not reduced to 0 and
we look for an element in the kernel of $(\gamma^A)_{*v}$ of the particular type $M=0$.
In this case, the equation (iii) is trivially satisfied and the equation (iv) may be reduced to
$Nb(\beta X+P)^*=0$.
We consider the singular value decomposition 
 of $b(\beta X+P)^*\in\K^{k\times k}$
 (for the quaternionic case see \cite{ZHANG}):
$$b(\beta X+P)^*=\mu
\begin{pmatrix}
Z_1&0
\cr
0&Z_2
\end{pmatrix}
\nu^*, \quad \mu,\nu\in\G(k),
$$
where $Z_1$ is diagonal without zero value on it and $Z_2=0\in \K^{r\times r}$.
As $\beta X+P$ is not invertible, we have  $r>0$. The existence of solutions in 
the equation $Nb(\beta X+P)^*=0$ is then equivalent to the existence of solutions in
$$N\mu \begin{pmatrix}
Z_1&0
\cr
0&Z_2
\end{pmatrix}=0
\iff
\mu^*N\mu \begin{pmatrix}
Z_1&0
\cr
0&Z_2
\end{pmatrix}=0.
$$
The fact that $r>0$ allows the choice of a non-zero, skew-symmetric matrix $\mu^*N\mu$ satisfying the last equation.
Thus $N\neq 0$ is skew-symmetric and
$(\gamma^A)_{*v}
\block{0}{N}=0$.
\end{proof}

\section{Properties}\label{subsec:propcayley}
This section consists of  the proof of Theorems \ref{thm:propcayley} and \ref{thm:cayleycontraction}. 
Recall the notations $x=\block{T}{P}\in \OO_{n,k}$
with $T\in\K^{(n-k)\times k}$,
$P\in \K^{k\times k}$
and the choice of $A=
\begin{pmatrix}
\alpha&T\cr
\beta&P
\end{pmatrix}\in\G(n)$.

\begin{proof}[Proof of \thmref{thm:propcayley}]\ 

\indent 1) First, we look at the independence of $\Gamma^x$ on the choice of $A$. 
With the notations of Remark \ref{rem:choice}, any matrix projecting on $x$ can be written as
$A{\scriptstyle{\bullet}} E=\begin{pmatrix}
\alpha E& T\cr
\beta E&P
\end{pmatrix}$
with  $E\in \G(n-k)$. An element
$v\in T_x \OO_{n,k}$ may be expressed as
$v=
A\block{X}{Y}=
\left(A{\scriptstyle{\bullet}} E\right)\block{E^*X}{Y}.$
The fact that $\Gamma^x$ does not depend
 on the choice of $A$ comes from $(\beta E)(E^*X)+P=\beta X+P$.

As for the injectivity, let 
$v_1=
A\block{X_1}{Y_1}$
and
$v_2=
A\block{X_2}{Y_2}
$
be two vectors of $\Gamma^x$. From Definition \ref{def:cayley}, the equality
$\gamma^A(v_1)=\gamma^A(v_2)$
is equivalent to the system
$$
\left\{\begin{array}{lrcl}
\text{(i)}& X_1b_1(\beta X_1+P)^*
&=&
X_2b_2(\beta X_2+P)^*,\\
\text{(ii)}&
b_1(\beta X_1+P)^*
&=&
b_2(\beta X_2+P)^*,

\end{array}
\right.
$$
from which we deduce
\begin{equation}\label{equa:X1X2}
(X_1-X_2)b_1(\beta X_1+P)^*=0.
\end{equation}
As $v_1\in\Gamma^x$ means that $\beta X_1+P$ is invertible, we get from \eqref{equa:X1X2} the equality $X_1=X_2$. 
Then the equation (ii) implies $b_1=b_2$, 
from which and \eqref{eq:b} we deduce  that $Y_1=Y_2$ and the injectivity of $\gamma^A$ on $\Gamma^x$.

Conversely, suppose $\gamma^A$ invertible on an open subset $U$. This implies the
injectivity of the differential
$(\gamma^A)_{*v}$ for any $v\in U$ and \propref{prop:derivative}
 gives the inclusion $U\subset \Gamma^x$.

\medskip

2) The values of $\tau$ and $\pi$ such that $\gamma^A(v)=\block{\tau}{\pi}$ are given by (\ref{equa:gammaA}). 
Let $\block{\tau}{\pi}$ such that $\pi+P^*$ is invertible. We are looking for matrices 
$X\in\K^{(n-k)\times k}$ and $Y\in\K^{k\times k}$, with $Y$ skew-symmetric, such that
$\beta X+P$ is invertible and the following system, which is equivalent to Definition \ref{def:cayley}, is satisfied:
$$\begin{array}{ll}
&\left\{
\begin{array}{lrcl}
\text{(a)}&\tau -\beta^*
&=&
-2Xb(\beta X+P)^*,
\\
\text{(b)}&\pi+P^*
&=&
2b(\beta X+P)^*.
\end{array}\right.\\
\end{array}
$$
In particular, we have from (b) that 
 the matrix $\pi+P^*$ is invertible if and only if $\beta X+P$ is so.
From (a) we get the value of $X$,
\begin{equation}\label{equa:leX}
X=-(\tau-\beta^*)(\pi+P^*)^{-1}.
\end{equation}
Also from (b) we obtain
$$b=\frac{1}{2}(\pi+P^*)\left[(\beta X +P)^*\right]^{-1}$$
and the expression of $Y$ follows from the fact that $Y$ is the skew-symmetric part of $b^{-1}$, that is, 
\begin{equation}\label{equa:ygrecq}
2Y=b^{-1}-(b^{-1})^*.
\end{equation} 
We need not the explicit expression.
If we replace those values in \eqref{equa:gammaA}, we get   $
\gamma^A(v)
=\block{\tau}{\pi}
$,
so we have proved the existence of a right inverse to the map $\gamma^A\colon \Gamma^x\to \Omega^x$. Since $\gamma^A$ is injective we obtain the desired result.
\end{proof}

\begin{remark}
For any Stiefel manifold it is possible to prove that the domain of injectivity $\Gamma^x$ is not the whole vector space $T_x\OO_{n,k}$.
\end{remark}

\begin{definition}\label{def:Cayleyopen}
For any $x=\block{T}{P}\in \OO_{n,k}$, the open subset
$$\Omega^x=
\left\{
\block{\tau}{\pi}\in \OO_{n,k}\mid \pi+P^* \text{ invertible}\right\}$$
is called 
\emph{a Cayley open subset of the Stiefel manifold.}
\end{definition}

We continue with  an explicit trivialization of the fibration $\rho$ over each Cayley open set.

\begin{proposition}\label{prop:localsection}
Let $x=\block{T}{P}\in \OO_{n,k}$  and let
$\Omega^x$ be 
 the open subset
of the elements 
$\block{\tau}{\pi}\in \OO_{n,k}$ such that $\pi+P^*$ is invertible.
Then the projection $\rho\colon \G(n)\to \OO_{n,k}$ admits a local section
$s^A\colon  
\Omega^x
\to \G(n)$.
\end{proposition}

\begin{proof}
With the identification $T_x\OO_{n,k}\cong (T_A\G(n-k))^{\bot}$, and from the definition of $\gamma^A$ we can write $\gamma^A=\rho\circ c_A$ on $\Gamma^x$. Moreover, we have proved in \thmref{thm:propcayley}
that the restriction
${\gamma^A}_{|\Gamma^x}\colon \Gamma^x\to \Omega^x$ is a diffeomorphism whose inverse 
is  denoted
$({\gamma^A}_{|\Gamma^x})^{-1}$.
We set
$$s^A=c_A\circ  ({\gamma^A}_{|\Gamma^x})^{-1}\colon \Omega^x\to \G(n)$$ and verify
$$\rho\circ s^A=\rho\circ c_A\circ ({\gamma^A}_{|\Gamma^x})^{-1}
={\gamma^A}_{|\Gamma^x}\circ ({\gamma^A}_{|\Gamma^x})^{-1}=\id_{\Omega^x}.$$
Notice that $s^A(\Omega^x)\subset \Omega(A^*)\cap \G(n)$.
\end{proof}

An explicit formula for $s^A$ could be obtained from those of $c_A$ and $({\gamma^A}_{|\Gamma^x})^{-1}$.

\begin{proof}[Proof of \thmref{thm:cayleycontraction}]
We choose $A=\begin{pmatrix}
\alpha&T\cr
\beta&P
\end{pmatrix}\in\G(n)$.
Let $s^A\colon  
\Omega^x
\to \G(n)$
be the local section of \propref{prop:localsection}. 
With the notations of the statement, we consider the application
$H\colon \Omega^x\times [0,1]\to \OO_{n,k}$
defined by
$$H(y,t)=\rho(c_A(tc_{A^*}(s^A(y)))).$$
This map verifies 
$H(y,0)=\rho(c_A(0))=\rho(A^*)=\block{\beta^*}{P^*}=\gamma^A(0)$
and
$H(y,1)=\rho(s^A(y))=y$.
Therefore, it is a contraction of $\Omega^x$ on the point $\gamma^A(0)$.
\end{proof}

\section{Some applications}\label{sec:appli}

\subsection{Lusternik-Schnirelmann category of some quaternionic Stiefel manifolds}
Let $\K=\HH$ be the algebra of quaternions, $\G(n)=\Sp(n)$ the symplectic group and $\OO_{n,k}=\XX_{n,k}$ the quaternionic Stiefel manifold. With the notations of the proof of \thmref{thm:cayleycontraction}, we observe that, in general, the point $\gamma^A(0)$ 
does not belong to $\Omega^x$. Therefore, our proof does not imply the contractibility of $\Omega^x$   
but only its contractibility in $\XX_{n,k}$, as stated.
This property suffices for our first application.

\smallskip
Recall that the Lusternik-Schnirelmann category (henceforth LS-category) of a topological space $X$ 
is 
the least integer $m \geq 0$ such that $X$ admits a covering by $(m+1)$ open sets which are
contractible in $X$, see \cite{CornLuptOprTan2003} for more details.  We denote it $\cat X$. 

\smallskip
The LS-category has applications in a wide range of fields coming from dynamical systems
to homotopy theory, but it has also proven to be
 difficult to determine. 
For instance, a longstanding problem is the computation of the
LS-category of Lie groups. In the case of unitary groups, W.~Singhof (\cite{Singhof75}) proved that
$\cat \u(n) = n$ by using an argument based on the eigenvalues. 
Nevertheless, this method cannot be carried out for the quaternionic group $\Sp(n)$, see \cite{MR2566482}.
However, some results have been obtained for small $n$ as
$\cat \Sp(2) = 3$ (\cite{MR0182969}) and $\cat \Sp(3) = 5$ (\cite{MR2022385}) and for the determination of some bounds as
 $\cat \Sp(n) \geq n + 2$ when $n \geq 3$ (\cite{MR2039767})
 and  $\cat \Sp(n) \leq \binom{n+1}{2}
 $ (\cite{MacPer2013}).
The quaternionic Stiefel manifolds $\XX_{n,k}$ are more accessible in certain ranges. 
For instance, we know that $\cat \XX_{n,k} =k$
when $n\geq 2k$. For proving that in \cite{Nishimoto2007} T. Nishimoto uses the number of
eigenvalues of a complex matrix in a way similar to Singhof's approach. 
This has also been established in \cite{KadzMimu2011}  by H. Kadzisa and M. Mimura 
from the Morse-Bott functions defined on $\XX_{n,k}$.
In the next proposition, we give a short proof of this result with the Cayley open subsets
of \defref{def:Cayleyopen}.

\begin{proposition}[\cite{Nishimoto2007}, \cite{KadzMimu2011}]
If $n\geq 2k$, we have $\cat \XX_{n,k}\geq 2k$.
\end{proposition}

\begin{proof}
Let $\theta\in]0,\pi/2[$ and take $x_{\theta}=\block{T_{\theta}}{P_{\theta}} \in \XX_{n,k}$, 
with  $P_{\theta}=(\cos \theta) I_k$ and $T_{\theta}=\block{0}{(\sin\theta) I_k}$.
We know from \thmref{thm:cayleycontraction} that $\Omega^{x_{\theta}}$ is contractible in $\XX_{n,k}$.

We choose $(k+1)$ numbers $\theta_i$ such that
$0<\theta_0<\theta_2<\dots<\theta_{k}<\pi/2$ and observe that an element 
$\pi\in \HH^{k\times k}$ such that $\pi+P_{\theta_i}$ is not invertible for all $i$
should have $(k+1)$ distinct real eigenvalues. This is impossible and the family 
$\left(\Omega^{x_{\theta_i}}\right)_{0\leq i\leq k}$ is an open cover of $\XX_{n,k}$
by  subsets contractible in~$\XX_{n,k}$.
\end{proof}

\subsection{Optimisation theory}
Let $\G(n)=\OO(n)$ be the orthogonal group and $\OO_{n,k}=\VV_{n,k}$ the real  Stiefel manifold. In optimisation theory, the problems with orthogonality constraints are widely known and have
concrete applications in many different areas (see \cite{ABSIL} for instance). A typical example is looking for $k$ orthogonal $n$-vectors that are optimal with respect to some parameter $f$ like cost or likehood.  This kind of problems can be seen as optimisation problems on a real Stiefel manifold. 

The most popular method for this study is the gradient descent method which can be summarized as follows. Let $x=x_0$ be an initial trial point in the Stiefel manifold $\VV_{n,k}$  and let $F$ be the negative gradient of $f$ at $x$. Then a curve
$\alpha(t)$ must be found on the manifold such that $\alpha(0)=x$ and $\alpha^\prime(0)=F$. By fixing a step size $\tau$ small enough,  the next iterate  is obtained by curvilinear search, that is, putting $x_1=\alpha(\tau)$. Under certain conditions the sequence $x_0, x_1,\dots$ converges to a local minimun of the function $f$.

Most existing methods either use matrix factorizations (such as the SVD decomposition) or 
require the determination of geodesic curves, which is computationally expensive. 
A different algorithm has been proposed in \cite{MR3127080}, where the curve is not a geodesic 
but  is constructed from the Cayley transform in the orthogonal group. 
Specifically, one considers the skew-symmetric matrix $A=Fx^*-xF^*$ and computes the Cayley transform
 $Q(t)=c_I(t A)$ \emph{on the group} $\OO(n)$. Since the group acts on the Stiefel manifold, the desired curve can be given by $\alpha(t)=Q(t)x$.

Our construction of a Cayley transform  is intrinsic to the Stiefel manifold and should lead to more efficient methods.

\section*{Acknowledgements}
This research was supported by Xunta de Galicia 2015-PG006, by ESF ACAT Grant 6442 and 
by the ANR-11-LABX-0007-01 ``CEMPI'' which gave us  the opportunity to work together. 
The three authors are also partially supported by the MINECO and 
FEDER research project MTM2016-78647-P.

\providecommand{\bysame}{\leavevmode\hbox to3em{\hrulefill}\thinspace}
\providecommand{\MR}{\relax\ifhmode\unskip\space\fi MR }
\providecommand{\MRhref}[2]{%
  \href{http://www.ams.org/mathscinet-getitem?mr=#1}{#2}
}
\providecommand{\href}[2]{#2}

\end{document}